\documentclass[12pt]{amsart}
\usepackage{amssymb}
\usepackage{t1enc}
\usepackage[mathscr]{eucal}
\usepackage{parskip}
\usepackage{xcolor}
\usepackage{tikz}
\usetikzlibrary{calc}
\usepackage{tikz-cd}
\usepackage{hyperref}
\numberwithin{equation}{section}
\usepackage[margin=2.9cm]{geometry}
\usepackage{booktabs}
\usepackage{enumerate}
\usepackage{bm}

\theoremstyle{plain}
\newtheorem{thm}{Theorem}[section]
\newtheorem{lem}[thm]{Lemma}
\newtheorem{cor}[thm]{Corollary}

 \theoremstyle{definition}
\newtheorem{defn}[thm]{Definition}
\newtheorem{rem}[thm]{Remark}

\newcommand{\mb}[1]{\mathbb{#1}}

\newcommand{\mr}[1]{\mathrm{#1}}

\renewcommand{\emptyset}{\varnothing}

\newcommand{\QT}{\operatorname{QType}}

\makeatletter
\@namedef{subjclassname@2020}{\textup{2020} Mathematics Subject Classification}
\makeatother

\begin{document}
\title[Bounding the signed count of real bitangents]{Bounding the signed count of real bitangents to plane quartics}

\author{Mario Kummer}
\address{Fakult\"at Mathematik \\ Institut f\"ur Geometrie \\ Technische Universit\"at Dresden}
\email{mario.kummer@tu-dresden.de}
\urladdr{tu-dresden.de/mn/math/geometrie/kummer/startseite/}

\author{Stephen McKean}
\address{Department of Mathematics \\ Harvard University} 
\email{smckean@math.harvard.edu}
\urladdr{shmckean.github.io}


\begin{abstract}
Using methods from enriched enumerative geometry, Larson and Vogt gave a signed count of the number of real bitangents to real smooth plane quartics. This signed count depends on a choice of a distinguished line. Larson and Vogt proved that this signed count is bounded below by 0, and they conjectured that the signed count is bounded above by 8. We prove this conjecture using real algebraic geometry, plane geometry, and some properties of convex sets.
\end{abstract}

\maketitle

\section{Introduction}
In this note, we prove an upper bound to the signed count of real bitangents to real smooth plane quartics introduced by Larson and Vogt~\cite{LV21}. Using tools from $\mb{A}^1$-enumerative geometry (also called \textit{quadratic} or \textit{enriched} enumerative geometry), Larson and Vogt define the sign $\QT_L(B)$ (relative to an auxiliary line $L\subset\mb{P}^2$) of a bitangent $B$ to a real smooth plane quartic $Q$. They then define the \textit{signed count}
\[s_L(Q):=\sum_{B\text{ real bitangent}}\QT_L(B).\]
If $L\cap Q(\mb{R})=\varnothing$, then $s_L(Q)=4$~\cite[Theorem 1]{LV21}. In general, $s_L(Q)$ is a non-negative even integer~\cite[Proposition 4.3]{LV21}. Larson and Vogt give examples of lines and quartics such that $s_L(Q)=0,2,4,6$, and 8. This leads to the conjecture that $s_L(Q)\in\{0,2,4,6,8\}$ for any choice of line and quartic~\cite[Conjecture 2]{LV21}. Using tools from tropical geometry, Markwig, Payne, and Shaw give a tropical criterion for quartics that satisfy $s_L(Q)\in\{0,2,4\}$~\cite[Theorem 5.2]{MPS22}.

We gather a few ideas from real algebraic geometry and plane geometry to prove $s_L(Q)\leq 8$ for any choice of plane quartic and auxiliary line. Paired with Larson and Vogt's lower bound and examples, it follows that $s_L(Q)\in\{0,2,4,6,8\}$. 

\begin{thm}\label{main theorem}
Let $Q\subset\mb{P}^2$ be a real smooth plane quartic. For any admissible line $L$, the signed count $s_L(Q)$ is at most 8. In particular, $s_L(Q)\in\{0,2,4,6,8\}$.
\end{thm}

\subsection*{Acknowledgements}
MK has been supported by the Deutsche Forschungsgemeinschaft under Grant No. 502861109. SM received support from an NSF MSPRF grant (DMS-2202825).

\section{A few results on plane geometry and convex sets}
In this section, we collect a few results on plane geometry and convex sets that we will need to prove Theorem~\ref{main theorem}.

\begin{lem}\label{lem:at least 2 boundary}
 Let $K\subset\mb{R}^n$ a compact set and $L\subset\mb{R}^n$ a line. If $|L\cap K|\geq 2$, then $|L\cap\partial K|\geq2$.
\end{lem}

\begin{proof}
  We identify $L$ with the real line $\mb{R}$. Then $L\cap K$ is a compact subset of $\mb{R}$ with at least two elements. In particular, its minimum and maximum are different points on the boundary of $L\cap K$. Thus $L$ intersects the boundary of $K$ in at least two points.
\end{proof}

\begin{cor}\label{cor:compact boundary}
 Let $K\subset\mb{R}^n$ a compact set and $L\subset\mb{R}^n$ a line. If $L\cap K^\circ\neq\emptyset$, then $|L\cap\partial K|\geq2$.
\end{cor}

\begin{proof}
    The condition $L\cap K^\circ\neq\emptyset$ implies that $|L\cap K|=\infty$, so we can apply the previous lemma.
\end{proof}

\begin{cor}\label{cor:bigon boundary}
 Let $I\subset\mb{R}^2$ be a line segment with endpoints $a$ and $b$. Let $C\subset\mb{R}^2$ be a Jordan arc with endpoints $a$ and $b$. Assume that $|I\cap C|$ is finite. Then any line $L\subset\mb{R}^2$ that meets $I$ also meets $C$.
\end{cor}

\begin{proof}
 Let $\varphi\colon [s,t]\to \mb{R}^2$ an injective continuous map with $\varphi(t_0)=a$ and $\varphi(t_1)=b$ whose image is $C$. Let $s=t_0<\cdots<t_m=t$ such that $\varphi^{-1}(I\cap C)=\{t_0,\ldots,t_m\}$. Choose $0\leq i<m$ such that $L$ intersects the line segment $I'$ with end points $\varphi(t_i)$ and $\varphi(t_{i+1})$. By replacing $I$ by $I'$ and $C$ by $\varphi([t_i,t_{i+1}])$ we can restrict to the case that $I\cap C$ consists of the endpoints of $I$ only. This means that $I\cup C$ is a Jordan curve. Without loss of generality we may assume $L$ intersects $I$ transversely in exactly one point and that this point is not on $C$. This implies that $L$ intersects the interior  (in the sense of the Jordan curve theorem) of the Jordan curve $I\cup C$ whose compact closure we denote by $K$. By Corollary~\ref{cor:compact boundary} $L$ intersects $\partial K=I\cup C$ in at least two points. Thus $L\cap C\neq\emptyset$ because $|I\cap L|=1$.
\end{proof}

\begin{lem}\label{lem:convextwo}
 Let $K\subset\mb{R}^n$ a compact convex set and $L\subset\mb{R}^n$ a line. If $L\cap K^\circ\neq\emptyset$, then $|L\cap\partial K|=2$.
\end{lem}

\begin{proof}
 By assumption $L\cap K$ is a line segment spanned by two different boundary points $v$ and $w$ of $K$ and there is some $u$ on $L\cap K^\circ$. Let $u\in U\subset K$ an open neighbourhood of $u$ in $K$. It remains to show that every point $u'$ on $L\cap K$ other than $v,w$ is in the interior of $K$. Without loss of generality, we can assume that $u'$ is in the line segment spanned by $v$ and $u$. Thus $u'=\mu v+(1-\mu)u$ for some $0\leq\mu<1$. The set
 \begin{equation*}
     U'=\{\mu v+(1-\mu)x\mid x\in U\}
 \end{equation*}
 is then an open neighbourhood of $u'$ in $K$.
\end{proof}

\begin{lem}\label{lem:quadri2}
 Let $Q\subset\mb{R}^2$ a convex quadrilateral. A line that intersects $Q$ intersects at least two of its edges.
\end{lem}

\begin{proof}
 Let $L$ be a line that intersects $Q$. If $L$ contains a vertex of $Q$, then it intersects the two edges that contain this vertex. Thus assume that $L$ does not contain a vertex of $Q$. This implies that $L$ intersects the interior of $Q$ and that it intersects each edge in at most one point. By Lemma \ref{lem:convextwo} it intersects the boundary of $Q$, which is the union of all edges, in two points. This implies the claim.
\end{proof}

\begin{lem}\label{lem:quadrilateral}
 Let $Q\subset\mb{R}^2$ a convex quadrilateral. Let $L$ be  a line that intersects both diagonals of $Q$. Then $L$ intersects two opposite edges of $Q$.
\end{lem}

\begin{proof}
 First consider the case that $L$ does not intersect the interior of $Q$. Then $L$ intersects $Q$ either in a single vertex or in an entire edge. In the former case $L$ intersects only one diagonal of $Q$. In the latter case $L$ intersects opposite edges. 
 
 Now assume that $L$ intersect the interior of $Q$. Then by Lemma \ref{lem:convextwo} the line $L$ intersects the boundary of $Q$ in exactly two different points $P_1$ and $P_2$. First assume that $P_1$ is a vertex of $Q$. Then $P_2$ cannot lie on an edge $E$ that contains $P_1$ because then $L$ intersects the interior of $Q$. Thus $L$ intersects $Q$ in at least three edges and therefore in particular in two opposite edges. 
 
 Finally assume that $P_1$ and $P_2$ both are not a vertex of $Q$. Then $P_1$ and $P_2$ lie on two different edges $E_1$ and $E_2$. Assume for the sake of a contradiction that $E_1$ and $E_2$ are not opposite. Then the convex hull of $E_1\cup E_2$ is a triangle $T$ with edges $E_1$, $E_2$ and $D$ where $D$ is a diagonal of $Q$. Because $L$ does not intersect any vertex of $T$, it intersects the interior of $T$. Therefore, the two points $P_1$ and $P_2$ are the only intersection points of $L$ with the boundary of $T$ by Lemma \ref{lem:convextwo}. In particular, our line $L$ does not intersect $D$.
\end{proof}

\begin{lem}\label{lem:quadrilateral cone}
Let $Q\subset\mb{R}^2$ be a convex quadrilateral. Let $E,E'\subset\partial Q$ be two opposite edges of $Q$. Let $x$ be the intersection of the diagonals of $Q$. Then the set of points that lie on some line passing through $E$ and $E'$ is equal to the union of $Q$ and the set of points that lie on some line passing through $E$ and $x$ (see Figure~\ref{fig:quadrilateral cone}).
\end{lem}
\begin{proof}
Since $Q$ is the convex hull of its four vertices, it is also the convex hull of $E$ and $E'$. Thus the latter set is contained in the former and it remains to show that if a point $y\in \mb{R}^2\backslash Q$ lies on a line $L$ through $E$ and $E'$, then $y$ lies on a line through $E$ and $x$.

Let $L'$ be the line through $x$ and $y$. We will show that $L'\cap E$ is nonempty. By assumption, $L'$ passes through $x$ and hence both diagonals of $Q$. Thus if $L'\cap E=\varnothing$, then Lemma~\ref{lem:quadrilateral} implies that $L'$ passes through the other pair of opposite edges of $Q$. Thus $L$ and $L'$ meet complementary pairs of opposite edges of $Q$, so $L$ and $L'$ must intersect within $Q$. But $y\in L\cap L'$ lies outside $Q$ by assumption. We thus either contradict this assumption or find that $|L\cap L'|\geq 2$, another contradiction.
\end{proof}

\begin{figure}
    \centering
    \begin{tikzpicture}
    \coordinate (A) at (0,0);
    \coordinate (B) at (4,0);
    \coordinate (C) at (3,2);
    \coordinate (D) at (1,3);
    \coordinate (E) at (-3/2,-1);
    \coordinate (F) at (5,-1);
    \coordinate (G) at (6,4);
    \coordinate (H) at (0,4);
    \coordinate (X) at (12/5,8/5);
    \draw[very thick,fill=black,fill opacity=.1] (A) -- (B) -- (C) -- (D) -- cycle;
    \fill[red,fill opacity=.1] (E) -- (X) -- (F) -- cycle;
    \fill[red,fill opacity=.1] (G) -- (X) -- (H) -- cycle;
    \filldraw[red] (X) circle (3pt);
    \draw[very thick,red] (A) -- (B);
    \draw[very thick,red] (C) -- (D);
    \node[xshift=-12pt] at (X) {$x$};
    \node[xshift=-12pt] at ($(A)!0.5!(D)$) {$Q$};
    \node[yshift=12pt] at ($(C)!0.5!(D)$) {$E$};
    \node[yshift=-12pt] at ($(A)!0.5!(B)$) {$E'$};
    \end{tikzpicture}
    \caption{Support of lines through opposite edges of a quadrilateral}
    \label{fig:quadrilateral cone}
\end{figure}
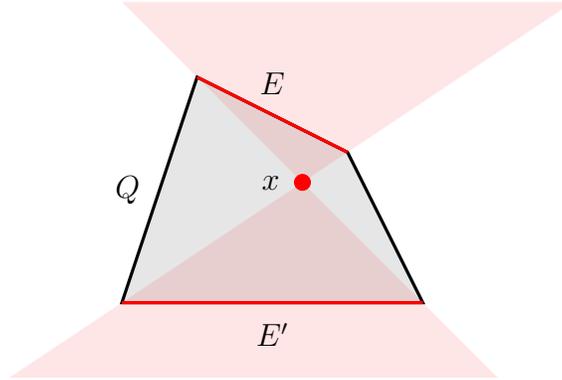

\section{Proving the conjecture}
Let $Q\subset\mb{P}^2$ a real smooth plane quartic curve. We say that a real line $L\subset\mb{P}^2$ is \emph{admissible} if it is disjoint from $Q\cap B$ for every real bitangent $B$ of $Q$. The QType~\cite[Definition 1.2]{LV21} of a real bitangent $B$ is the sign by which $B$ should be counted.

\begin{defn}
Let $f\in\mb{R}[x_0,x_1,x_2]$ be a homogeneous degree 4 polynomial such that $\mb{V}(f)=Q$. Let $B$ be a bitangent to $Q$, and denote $Q\cap B=2Z$ for $Z=z_1+z_2$ a degree 2 divisor. Let $L$ be an admissible real line. Denote by $\partial_L$ a derivation with respect to a linear form vanishing along $L$. The \textit{QType of $B$ with respect to $L$} is defined to be
\[\QT_L(B):=\mr{sign}(\partial_L f(z_1)\cdot\partial_L f(z_2))\in\{+1,-1\}.\]
\end{defn}

By \cite[Equ.~(6)]{LV21} a real bitagent can have negative $\QT$ only if it intersects $Q$ in real points only. In this case we say that the bitangent is \emph{split}.
Geometrically, a split bitangent $B$ has QType $+1$ if all connected components of $Q(\mb{R})$ that it meets lie in the same component of $\mb{R}^2\backslash B$ and QType $-1$ if it separates the connected components of $Q(\mb{R})$ that it meets (see Lemma~\ref{lem:signviahalfplanes}). Here we identify $\mb{R}^2=(\mb{P}^2\setminus L)(\mb{R})$.

\begin{rem}
A priori, $\QT_L(B)$ depends on the choice of the polynomial $f$ cutting out $Q$, as well as on the choice of linear form $\ell$ determining $\partial_L$. However, changing either of these choices changes both $\partial_L f(z_i)$ by the same non-zero scalar. This shows that
$\mr{sign}(\partial_L f(z_1)\cdot\partial_L f(z_2))$ is well-defined.
\end{rem}

For an admissible line $L\subset\mb{P}^2$ let
\begin{equation*}
 s_L(Q):=\sum_{B\textnormal{ real bitangent}} \QT_L(B)
\end{equation*}
denote the signed count (relative to $L$) of bitangents.

\begin{defn}
Let $L\subset\mb{P}^2$ a real line and $B$ a split bitangent such that $B\cap L\cap Q=\emptyset$. The \emph{grate} $g_L(B)$ of $B$ with respect to $L$ is the convex hull of $B\cap Q$ in $\mb{R}^2=(\mb{P}^2\setminus L)(\mb{R})$. We say that the grate of $B$ is \emph{positive} resp. \emph{negative} if $\QT_L(B)=+1$ or $\QT_L(B)=-1$ respectively.
\end{defn}

The statement of the next lemma can be found in \cite[p.~15]{LV21}.
\begin{lem}\label{lem:gratesformula0}
Let $B$ be a split bitangent and $L_1,L_2$ two real lines with $B\cap L_i\cap Q=\emptyset$ for $i=1,2$. Then one has
\begin{equation*}
\QT_{L_2}(B)=\begin{cases}
\QT_{L_1}(B)& \textnormal{ if } L_2\cap g_{L_1}(B)=\emptyset\\
-\QT_{L_1}(B)& \textnormal{ if } L_2\cap g_{L_1}(B)\neq\emptyset.
\end{cases}
\end{equation*}
\end{lem}

From this one can deduce the following, which is \cite[Equ.~(7)]{LV21}.

\begin{lem}\label{lem:gratesformula}
 Let $L,L'\subset\mb{P}^2$ two admissible lines. For $\epsilon\in\{\pm1\}$ we denote by $\sigma(L,L',\epsilon)$ the number of split bitangents $B$ with $\QT_L(B)=\epsilon$ such that $L'$ intersects $g_L(B)$. Then
 \begin{equation*}
  s_{L'}(Q)=s_L(Q)-2\cdot(\sigma(L,L',+1)-\sigma(L,L',-1)).
 \end{equation*}
\end{lem}

From now on we fix an admissible line $L_0$ which is disjoint from $Q(\mb{R})$. By \cite[Theorem~1]{LV21} we have $s_{L_0}(Q)=4$. Thus for any admissible line $L$ we have that $s_L(Q)$ is even by Lemma \ref{lem:gratesformula}. Since by \cite[Proposition~4.3]{LV21} we also have $s_L(Q)\geq0$, it only remains to show that $s_L(Q)\leq8$. By Lemma \ref{lem:gratesformula} this is equivalent to proving that
\begin{equation}\label{eq:sign}
 \sigma(L_0,L,-1)-\sigma(L_0,L,+1)\leq2.
\end{equation}

The general strategy proceeds in two cases. First, if our quartic has at most 8 real bitangents, then Theorem~\ref{main theorem} is trivially true. Second, in the case that $Q$ has more than 8 real bitangents, we will, for each two connected components $Q_1,Q_2$ of $Q(\mb{R})$, study the split bitangents meeting $Q_1$ and $Q_2$ separately. We first show that there are exactly two positive and two negative grates connecting $Q_1$ and $Q_2$. We will then show that if a line meets among those four grates more negative grates than positive grates, it must meet $Q_1$ and $Q_2$. By B\'ezout's theorem, this line does not meet any other components of $Q(\mb{R})$. This yields the desired upper bound.

To begin, we characterize split bitangents determining negative grates as those separating a pair of connected components of $Q(\mb{R})$.

\begin{lem}\label{lem:signviahalfplanes}
Let $B$ be a split bitangent and $L$ a real line with $L\cap B\cap Q=\emptyset$. Let $V_1,V_2$ the closures of the two connected components of $(\mb{P}^2\setminus(L\cup B))(\mb{R})$. Then every connected component of $Q(\mb{R})$ is contained in $V_1$ or in $V_2$. We have $\QT_{L}(B)=-1$ if and only if $B$ intersects a connected component contained in $V_1$ and one contained in $V_2$.
\end{lem}

\begin{proof}
 It is clear that if $B$ intersects $Q$ in only one point, then $\QT_{L}(B)=+1$. Thus we may assume that $B$ intersects $Q(\mb{R})$ in two distinct points $z_1$ and $z_2$. Let $P\in(\mb{P}^2\setminus L)(\mb{R})$ any point that is not contained in the interior of any oval of $Q(\mb{R})$ (i.e. $P$ is contained in the non-orientable connected component of $(\mb{P}^2\setminus Q)(\mb{R})$). We identify $\mb{R}^2$ with $(\mb{P}^2\setminus L)(\mb{R})$ and let $F$ a polynomial on $\mb{R}^2$ defining $Q$. Without loss of generality we can assume that $F(P)<0$. Then the normal vector of $Q$ at $z_i$, which is the gradient of $F$ at $z_1$, points towards the interior of the component $Q_i$ containing $z_i$. Thus the normal vectors at $z_1$ and $z_2$ point towards different connected components of $\mb{R}^2\setminus B$ if and only if $V_1$ and $V_2$ each contain one of the components $Q_1$ and $Q_2$.
\end{proof}

\begin{cor}\label{cor:split one component}
Let $L$ be a split bitangent to a real smooth plane quartic $Q$. If $L$ only meets one connected component of $Q(\mb{R})$, then the grate associated to $L$ is positive. 
\end{cor}
\begin{proof}
By Lemma~\ref{lem:signviahalfplanes}, the QType of a split bitangent can only be negative if the split bitangent meets two distinct connected components of $Q(\mb{R})$.
\end{proof}

We may thus restrict our attention to split bitangents meeting two distinct connected components of $Q(\mb{R})$. 
Let $s$ be the number of connected components of $Q(\mb{R})$ and let $a=1$ if $Q(\mb{C})\setminus Q(\mb{R})$ is connected and $a=0$ otherwise. By Harnack's theorem one has $0\leq s\leq 4$. If $s=4$, then $a=0$ and if $s=0$, then $a=1$. Finally, if $s$ is odd, then $a=1$ \cite[Prop.~3.1]{grossharris}.
The total number of real bitangents equals
$4\cdot (2^{s-1} - 1 + a)$ if $s>0$ and $4$ if $s=0$ \cite[Prop.~5.1]{grossharris}.
This implies in particular that if $s<3$ there are not more than $8$ real bitangents. Thus we can further restrict our attention to the case $s\geq3$. Then the next lemma shows that there are always four bitangents meeting any pair of two distinct connected components. Moreover, these four split bitangents determine two positive grates and two negative grates.

\begin{lem}\label{lem:four grates}
 Assume that $s\geq3$.
 Let $Q_1,Q_2$ be two distinct connected components of $Q(\mb{R})$. The following are true:
 \begin{enumerate}[(a)]
     \item There are exactly four split bitangents that meet both $Q_1$ and $Q_2$.
     \item With respect to $L_0$ the four split bitangents meeting $Q_1$ and $Q_2$ determine two positive grates $G_1^+,G_2^+$ and two negative grates $G_1^-,G_2^-$.
     \item  The two positive grates $G_1^+,G_2^+$ are edges of the convex hull of $Q_1\cup Q_2$.
     \item The two negative grates $G_1^-,G_2^-$ intersect.
 \end{enumerate}
\end{lem}

\begin{proof}
 Let $Q_1,\ldots,Q_s$ the connected components of $Q(\mb{R})$.
 We first claim that there is a line $L_1$ which does not intersect $Q(\mb{R})$ such that $Q_1$ and $Q_2$ are contained in two different connected components of $(\mb{P}^2\setminus (L_0\cup L_1))(\mb{R})$. Let $\omega_0$ be a linear form whose zero set is $L_0$. Since $Q$ is embedded canonically, we can regard $\omega_0$ as a holomorphic differential on $Q$ which has no real zeros. As such it induces an orientation $\Omega_0$ of $Q(\mb{R})$.  \cite[Cor.~2.2]{tottheta} says that every but at most one orientation of $Q(\mb{R})$ which agrees on $Q_1$ with $\Omega_0$ is induced by a holomorphic differential on $Q$ without real zeros. Since $s\geq3$ there are at least two orientations on $Q(\mb{R})$ that agree on $Q_1$ with $\Omega_0$ but differ from $\Omega_0$ on $Q_2$. Thus at least one of them is induced by a holomorphic differential $\omega_1$ without real zeros. This means that the rational function $\frac{\omega_0}{\omega_1}$ is positive on $Q_1$ and negative on $Q_2$. Thus $Q_1$ and $Q_2$ are contained in two different connected components of $(\mb{P}^2\setminus (L_0\cup L_1))(\mb{R})$ where $L_1$ is the line defined as the zeros set of $\omega_1$ regarded as linear form on $\mb{P}^2$. 

 Now we consider $Q(\mb{R})$ as a subset of $\mb{R}^2=(\mb{P}^2\setminus L_0)(\mb{R})$.
 We have shown that the components $Q_1$ and $Q_2$ are strictly separated by the line $L_1$ in $\mb{R}^2$.
 Thus we can apply \cite[Prop.~3.2]{tottheta} to the convex hulls of $Q_1$ and $Q_2$ which shows that there are (up to a nonzero scalar multiple) exactly two affine linear functions that are nonnegative on $Q_1\cup Q_2$ but zero in at least one point on each $Q_1$ and $Q_2$. Thus there are exactly two split bitangents $B$ with $\QT(B)=+1$ that meet both $Q_1$ and $Q_2$. Their grates are necessarily edges of the convex hull of $Q_1\cup Q_2$. 
 
For proving $(a)$, $(b)$ and $(c)$ it remains to show that the total number of real bitangents meeting $Q_1$ and $Q_2$ is four. Recall that a \emph{semi-orientation} is an equivalence class of orientations modulo global reversion. Thus there are exactly $2^{s-1}$ semiorientations on $Q(\mb{R})$. First consider the case $s=3$. Then we necessarily have $a=1$ and \cite[Thm.~6.9]{kummer2023signed} says that the number of real bitangents which meet $Q_1$ and $Q_2$ is the same as the number of semi-orientations on $Q(\mb{R})$. Since $s=3$, there are four semi-orientations which implies the claim.
 Now consider the case $s=4$. Then we necessarily have $a=0$ and \cite[Thm.~6.11]{kummer2023signed} says that the number of real bitangents which meet $Q_1$ and $Q_2$ is the same as the number of semi-orientations on $Q(\mb{R})$ that restrict to a certain fixed semi-orientation on $Q_3\cup Q_4$. Since $s=4$, there are four such semi-orientations which implies the claim.

 For part $(d)$ let $\alpha_i$ be a linear form whose zero set is the bitangent $B_i$ such that $g_{L_0}(B_i)=G_i^-$ for $i=1,2$. Since $\QT_{L_0}(B_i)=-1$, we have by Lemma \ref{lem:signviahalfplanes} that $Q_1$ and $Q_2$ are contained in the closure of different connected components of $(\mb{P}^2\setminus (L_0\cup B_i))(\mb{R})$. After replacing $\alpha_i$ by $-\alpha_i$ if necessary, we can thus assume that $\frac{\alpha_i}{\omega_0}$ is nonnegative on $Q_1$ and nonpositive on $Q_2$. Thus $\frac{\alpha_1}{\alpha_2}$ is nonnegative on $Q_1\cup Q_2$. This shows that $Q_1$ and $Q_2$ are contained in the closure of the same connected component of $(\mb{P}^2\setminus (B_1\cup B_2))(\mb{R})$ which implies $\QT_{B_1}(B_2)=+1$ and $\QT_{B_2}(B_1)=+1$. By Lemma \ref{lem:gratesformula0} and because $\QT_{L_0}(B_i)=-1$ for $i=1,2$, this shows that $B_1\cap g_{L_0}(B_2)\neq\emptyset$ and $B_2\cap g_{L_0}(B_1)\neq\emptyset$. Since $B_1$ and $B_2$ intersect in exactly one point, we must have $g_{L_0}(B_1)\cap g_{L_0}(B_2)\neq\emptyset$.
\end{proof}

\begin{figure}
    \centering
    \begin{tikzpicture}[scale=1.5]
    \def\H{30pt}
    \def\angle{40};
    \draw[very thick,cm={cos(45) ,-sin(45) ,sin(45) ,cos(45) ,(40pt,0)}]
  plot[domain=-pi:pi,samples=200]
  ({.78*\H *cos(\x/4 r)*sin(\x r)},{-\H*(cos(\x r))})
  -- cycle;
  \draw[very thick,cm={cos(-45) ,-sin(-45) ,sin(-45) ,cos(-45) ,(-40pt,0)}]
  plot[domain=-pi:pi,samples=200]
  ({.78*\H *cos(\x/4 r)*sin(\x r)},{-\H*(cos(\x r))})
  -- cycle;
  \coordinate (A) at (90pt,25.1pt);
  \coordinate (B) at (-90pt,25.1pt);
  \coordinate (C) at (90pt,-26.8pt);
  \coordinate (D) at (-90pt,-26.8pt);
  \coordinate (X) at (0pt,-5.35pt);
  \draw[thick,blue,dotted] (A) -- (B);
  \draw[thick,blue,dotted] (C) -- (D);
  \draw[thick,red,dotted] (X) --++ (\angle:80pt);
  \draw[thick,red,dotted] (X) --++ (\angle-180:70pt);
  \draw[thick,red,dotted] (X) --++ (-\angle-180:80pt);
  \draw[thick,red,dotted] (X) --++ (-\angle:70pt);
  \coordinate (A1) at (50pt,25.1pt);
  \coordinate (B1) at (-50pt,25.1pt);
  \coordinate (C1) at (33pt,-26.8pt);
  \coordinate (D1) at (-33pt,-26.8pt);
  \filldraw[blue] (A1) circle (2pt);
  \filldraw[blue] (B1) circle (2pt);
  \filldraw[blue] (C1) circle (2pt);
  \filldraw[blue] (D1) circle (2pt);
  \draw[very thick,blue] (A1) -- (B1);
  \draw[very thick,blue] (C1) -- (D1);
  \coordinate (E) at ($(X)+1.1*({cos(\angle)},{sin(\angle)})$);
  \coordinate (F) at ($(X)+1.1*({cos(-\angle-180)},{sin(-\angle-180)})$);
  \coordinate (G) at ($(X)-.9*({cos(\angle)},{sin(\angle)})$);
  \coordinate (H) at ($(X)-.9*({cos(-\angle-180)},{sin(-\angle-180)})$);
  \filldraw[red] (E) circle (2pt);
  \filldraw[red] (F) circle (2pt);
  \filldraw[red] (G) circle (2pt);
  \filldraw[red] (H) circle (2pt);
  \draw[very thick,red] (G) -- (E);
  \draw[very thick,red] (H) --(F);
  \node at (-75pt,0pt) {$Q_1$};
  \node at (75pt,0pt) {$Q_2$};
  \node[blue] at (0,35pt) {$G_1^+$};
  \node[blue] at (0,-35pt) {$G_2^+$};
  \node[red] at (-10pt,15pt) {$G_1^-$};
  \node[red] at (10pt,15pt) {$G_2^-$};
    \end{tikzpicture}
    \caption{Grates meeting two components}
    \label{fig:convex hull and grates}
\end{figure}
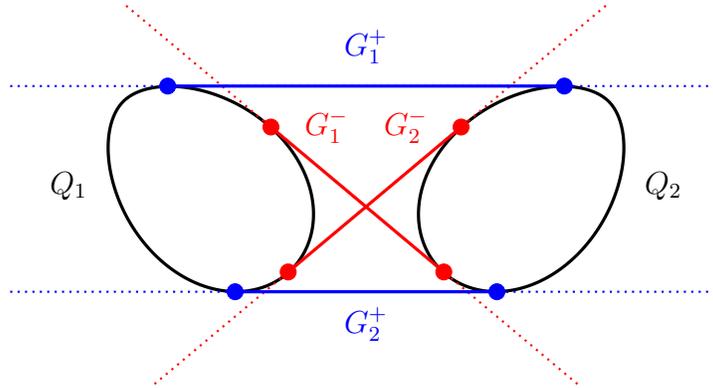

Given two connected components $Q_1,Q_2$ of $Q(\mb{R})$, we will consistently denote by $G_1^+,G_2^+$ and $G_1^-,G_2^-$ the positive and negative grates, respectively, determined by the split bitangents meeting both $Q_1$ and $Q_2$.

Each split bitangent of $Q_1,Q_2$ that determines a negative grate intersects each split bitangent that determines a positive grate. In fact, these intersections occur within the positive grates.

\begin{lem}\label{lem:intersect in positive grate}
Let $B_1,B_2$ be the split bitangents containing $G_1^-,G_2^-$. Then $|B_i\cap G_j^+|=1$ for $i,j\in\{1,2\}$.
\end{lem}
\begin{proof}
Any two bitangents to $Q$ are distinct, so we just need to show that $|B_i\cap G_j^+|>0$. We will call a split bitangent \textit{positive} (respectively \textit{negative}) if it determines a positive (respectively negative) grate. Each positive grate touches each $Q_i$ once, so the region $R$ between the positive grates and $Q_1\cup Q_2$ is bounded by the Jordan curve theorem. The negative bitangents cannot cross the components $Q_i$, so the negative bitangents can only enter or exit $R$ through the positive grates.

By Lemma~\ref{lem:signviahalfplanes}, a split bitangent $L$ meeting two components $Q_1,Q_2$ is negative if the regions bounded by $Q_1$ and $Q_2$ lie in different components of $\mb{R}^2\backslash L$. In particular, a negative bitangent passes through the interior of the convex hull $K$ of $Q_1\cup Q_2$. By Lemma~\ref{lem:convextwo}, $|B_i\cap \partial K|=2$. The boundary $\partial K$ consists of the positive grates $G_1^+,G_2^+$, and two other boundary components  that are separated by the line segments $S_i$ connecting $G_1^+\cap Q_i$ and $G_2^+\cap Q_i$ (see Figure~\ref{fig:boundary}). Since $L$ cannot cross $Q_i$, negative bitangents cannot cross $S_i$ and hence $B_i\cap\partial K\subset G_1^+\cup G_2^+$. Thus $|B_i\cap G_1^+|+|B_i\cap G_2^+|=2$, so $|B_i\cap G_j^+|=1$.
\end{proof}

As a corollary, we find that the negative grates are contained in the interior of the convex hull of the positive grates.

\begin{cor}\label{cor:interior of positive grates}
The negative grates $G_1^-,G_2^-$ are contained in the convex hull of $G_1^+\cup G_2^+$.
\end{cor}
\begin{proof}
Lemma~\ref{lem:intersect in positive grate} states that the split bitangents containing $G_1^-$ and $G_2^-$ intersect the positive grates, so $G_1^-,G_2^-$ are contained in the convex hull of $G_1^+\cup G_2^+$.
\end{proof}

\begin{lem}\label{lem:segmentthencurve}
 Let $Q_0$ a connected component of $Q(\mb{R})\subset\mb{R}^2=(\mb{P}^2\setminus L_0)(\mb{R})$. Let $a,b\in Q_0$ two different points and $S$ the line segment spanned by $a$ and $b$. If a line intersects $S$, then it intersects $Q_0$.
\end{lem}

\begin{proof}
Let $J$ the closure of one of the two connected components of $Q_0\setminus \{a,b\}$. This is a Jordan arc with the same endpoints as $S$. Thus by Corollary~\ref{cor:bigon boundary}, any line meeting $S$ also meets $J\subset Q_0$.
\end{proof}

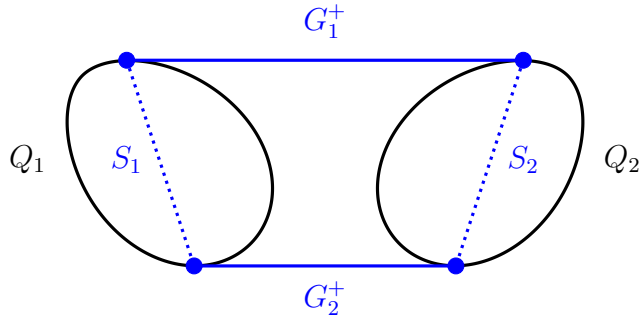
\begin{figure}
    \centering
    \begin{tikzpicture}[scale=1.5]
    \def\H{30pt}
    \def\angle{40};
    \draw[very thick,cm={cos(45) ,-sin(45) ,sin(45) ,cos(45) ,(40pt,0)}]
  plot[domain=-pi:pi,samples=200]
  ({.78*\H *cos(\x/4 r)*sin(\x r)},{-\H*(cos(\x r))})
  -- cycle;
  \draw[very thick,cm={cos(-45) ,-sin(-45) ,sin(-45) ,cos(-45) ,(-40pt,0)}]
  plot[domain=-pi:pi,samples=200]
  ({.78*\H *cos(\x/4 r)*sin(\x r)},{-\H*(cos(\x r))})
  -- cycle;
  \coordinate (A) at (90pt,25.1pt);
  \coordinate (B) at (-90pt,25.1pt);
  \coordinate (C) at (90pt,-26.8pt);
  \coordinate (D) at (-90pt,-26.8pt);
  \coordinate (X) at (0pt,-5.35pt);
  \coordinate (A1) at (50pt,25.1pt);
  \coordinate (B1) at (-50pt,25.1pt);
  \coordinate (C1) at (33pt,-26.8pt);
  \coordinate (D1) at (-33pt,-26.8pt);
  \filldraw[blue] (A1) circle (2pt);
  \filldraw[blue] (B1) circle (2pt);
  \filldraw[blue] (C1) circle (2pt);
  \filldraw[blue] (D1) circle (2pt);
  \draw[very thick,blue] (A1) -- (B1);
  \draw[very thick,blue] (C1) -- (D1);
  \draw[very thick,dotted,blue] (A1) -- (C1);
  \draw[very thick,dotted,blue] (B1) -- (D1);
  \node[blue] at (0,35pt) {$G_1^+$};
  \node[blue] at (0,-35pt) {$G_2^+$};
  \node[blue] at (-50pt,0) {$S_1$};
  \node[blue] at (50pt,0) {$S_2$};
  \node at (-75pt,0pt) {$Q_1$};
  \node at (75pt,0pt) {$Q_2$};
    \end{tikzpicture}
    \caption{Quadrilateral from positive grates}
    \label{fig:boundary}
\end{figure}

We now show that if a line meets a negative grate, it meets at least two of $Q_1,Q_2,G_1^+,G_2^+$.

\begin{lem}\label{lem:meeting a negative grate}
Let $L\subset\mb{P}^2$ be an admissible line. If $L$ meets one of the negative grates of $Q_1,Q_2$, then $L$ meets at least two of $Q_1,Q_2,G_1^+,G_2^+$. 
\end{lem}

\begin{proof}
Let $R$ be the convex hull of $G_1^+\cup G_2^+$. Note that $R$ is a convex quadrilateral with edges $S_1$, $S_2$, $G_1^+$ and $G_2^+$, where $S_i$ is the line segment connecting $G_1^+\cap Q_i$ and $G_2^+\cap Q_i$ (see Figure~\ref{fig:boundary}). If $L$ meets a negative grate, then $L$ meets at least two edges of $R$ by Lemma~\ref{lem:quadri2}. By Lemma~\ref{lem:segmentthencurve} any line meeting $S_i$ also meets $Q_i$. Thus if $L$ meets at least two of $S_1,S_2,G_1^+,G_2^+$, then $L$ meets at least two of $Q_1,Q_2,G_1^+,G_2^+$.
\end{proof}

As a final ingredient, we show that if a line meets both negative grates, it meets either both $Q_1$ and $Q_2$ or both $G_1^+$ and $G_2^+$.

\begin{lem}\label{lem:both negative grates}
Let $L\subset\mb{P}^2$ be an admissible line. Assume that $L$ meets both $G_1^-$ and $G_2^-$. Then either $L$ meets both $Q_1$ and $Q_2$, or $L$ meets both $G_1^+$ and $G_2^+$.
\end{lem}
\begin{proof}
See Figure~\ref{fig:two negative grates} for an illustration of some of the objects appearing in this proof. Let $R$ be the convex hull of $G_1^-\cup G_2^-$. The negative grates intersect (necessarily within $R^\circ)$ by Lemma~\ref{lem:four grates}, which implies that the negative grates $G_1^-,G_2^-$ are the diagonals of the convex quadrilateral $R$. Let 
\begin{itemize}
    \item $E$ be the line segment connecting $G_1^-\cap Q_1$ and $G_2^-\cap Q_2$,
    \item $E'$ be the line segment connecting $G_1^-\cap Q_2$ and $G_2^-\cap Q_1$,
    \item $F$ be the line segment connecting $G_1^-\cap Q_1$ and $G_2^-\cap Q_1$, and
    \item $F'$ be the line segment connecting $G_1^-\cap Q_2$ and $G_2^-\cap Q_2$.
\end{itemize}
These are the four edges of $R$ and by Lemma~\ref{lem:quadrilateral}, $L$ either passes through both $E,E'$ or through both $F,F'$. If $L$ passes through both $F$ and $F'$, then by Lemma~\ref{lem:segmentthencurve} $L$ also intersects $Q_1$ and $Q_2$.

Now suppose that $L$ passes through both $E$ and $E'$. Let $x=G_1^-\cap G_2^-$ be the intersection of the diagonals of $R$. Let $X$ be the set of all points lying on some line through $x$ and $E$. Since $L$ intersects $E,E'$, Lemma~\ref{lem:quadrilateral cone} implies that $L$ is contained in $X\cup R$. By Lemma~\ref{lem:intersect in positive grate}, the positive grates $G_1^+,G_2^+$ bound the quadrilateral $R$ within $X\cup R$. Any line in $X\cup R$ must pass through this bounded portion, so $L$ must pass through $G_1^+$ and $G_2^+$.
\end{proof}

\begin{figure}
    \centering
    \begin{tikzpicture}[scale=1.5]
    \def\H{30pt}
    \def\angle{40};
    \draw[very thick,cm={cos(45) ,-sin(45) ,sin(45) ,cos(45) ,(40pt,0)}]
  plot[domain=-pi:pi,samples=200]
  ({.78*\H *cos(\x/4 r)*sin(\x r)},{-\H*(cos(\x r))})
  -- cycle;
  \draw[very thick,cm={cos(-45) ,-sin(-45) ,sin(-45) ,cos(-45) ,(-40pt,0)}]
  plot[domain=-pi:pi,samples=200]
  ({.78*\H *cos(\x/4 r)*sin(\x r)},{-\H*(cos(\x r))})
  -- cycle;
  \coordinate (A) at (90pt,25.1pt);
  \coordinate (B) at (-90pt,25.1pt);
  \coordinate (C) at (90pt,-26.8pt);
  \coordinate (D) at (-90pt,-26.8pt);
  \coordinate (X) at (0pt,-5.35pt);
  \coordinate (A1) at (50pt,25.1pt);
  \coordinate (B1) at (-50pt,25.1pt);
  \coordinate (C1) at (33pt,-26.8pt);
  \coordinate (D1) at (-33pt,-26.8pt);
  \filldraw[blue] (A1) circle (2pt);
  \filldraw[blue] (B1) circle (2pt);
  \filldraw[blue] (C1) circle (2pt);
  \filldraw[blue] (D1) circle (2pt);
  \draw[very thick,blue] (A1) -- (B1);
  \draw[very thick,blue] (C1) -- (D1);
  \coordinate (E) at ($(X)+1.1*({cos(\angle)},{sin(\angle)})$);
  \coordinate (F) at ($(X)+1.1*({cos(-\angle-180)},{sin(-\angle-180)})$);
  \coordinate (G) at ($(X)-.9*({cos(\angle)},{sin(\angle)})$);
  \coordinate (H) at ($(X)-.9*({cos(-\angle-180)},{sin(-\angle-180)})$);
  \coordinate (E1) at ($(X)+2.5*1.1*({cos(\angle)},{sin(\angle)})$);
  \coordinate (F1) at ($(X)+2.5*1.1*({cos(-\angle-180)},{sin(-\angle-180)})$);
  \coordinate (G1) at ($(X)-2.5*.9*({cos(\angle)},{sin(\angle)})$);
  \coordinate (H1) at ($(X)-2.5*.9*({cos(-\angle-180)},{sin(-\angle-180)})$);
  \draw[red,very thick,dotted] (E) -- (F) -- (G) -- (H) -- cycle;
  \fill[green,opacity=.2] (E1) -- (F1) -- (X) -- (G1) -- (H1) -- (X) -- cycle;
  \filldraw[red] (E) circle (2pt);
  \filldraw[red] (F) circle (2pt);
  \filldraw[red] (G) circle (2pt);
  \filldraw[red] (H) circle (2pt);
  \draw[very thick,red] (G) -- (E);
  \draw[very thick,red] (H) --(F);
  \node at (-75pt,0pt) {$Q_1$};
  \node at (75pt,0pt) {$Q_2$};
  \node[black!60!green] at (0,38pt) {$X$};
  \node[red] at (0,10pt) {$E$};
  \node[red] at (0,-17pt) {$E'$};
  \node[red] at (-28pt,-3pt) {$F$};
  \node[red] at (28pt,-3pt) {$F'$};
    \end{tikzpicture}
    \caption{Meeting two negative grates}
    \label{fig:two negative grates}
\end{figure}
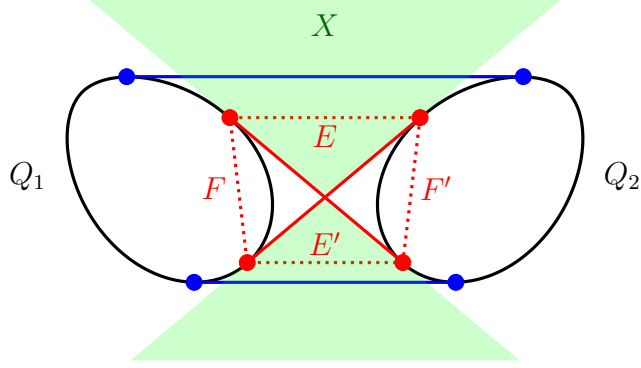

We can now prove Theorem~\ref{main theorem}.

\begin{proof}[Proof of Theorem~\ref{main theorem}]
Recall that we are trying to show that
\[\sigma(L_0,L,-1)-\sigma(L_0,L,+1)\leq 2,\]
where $L_0$ is a line disjoint from $Q(\mb{R})$ and $L$ is any admissible line (see inequality~\ref{eq:sign}). In other words, we want to show that any admissible line can meet at most 2 more negative grates than positive grates. 
Let $Q_1,\ldots,Q_s\subset Q(\mb{R})$ be the $s\geq 3$ distinct connected components. For $1\leq i<j\leq s$ and $\epsilon\in\{\pm1\}$ we denote by $\sigma_{ij}(L_0,L,\epsilon)$ the number of split bitangents $B$ with $\QT_{L_0}(B)=\epsilon$ such that $L$ intersects $g_{L_0}(B)$ and such that $B$ intersects both $Q_i$ and $Q_j$. Further let $\sigma_{ij}(L_0,L)=\sigma_{ij}(L_0,L,-1)-\sigma_{ij}(L_0,L,+1)$. By Corollary~\ref{cor:split one component} we have
\begin{equation}\label{eq:sumoverpairs}
    \sigma(L_0,L,-1)-\sigma(L_0,L,+1)\leq\sum_{1\leq i< j\leq s}\sigma_{ij}(L_0,L).
\end{equation}
For each $i\neq j$, let $G_{ij}^+,G_{ij}^{'+}$ be the positive grates and $G_{ij}^-,G_{ij}^{'-}$ be the negative grates associated to $Q_i$ and $Q_j$. We have $\sigma_{ij}(L_0,L)>0$ only if
\begin{enumerate}[(i)]
\item $L$ meets exactly one of $G_{ij}^-,G_{ij}^{'-}$ and neither of $G_{ij}^+,G_{ij}^{'+}$, or
\item $L$ meets both $G_{ij}^-,G_{ij}^{'-}$ and at most one of $G_{ij}^+,G_{ij}^{'+}$.
\end{enumerate}
In case (i), $L$ meets both $Q_i$ and $Q_j$ by Lemma~\ref{lem:meeting a negative grate}. In case (ii), $L$ meets both $Q_i$ and $Q_j$ by Lemma~\ref{lem:both negative grates}. Thus any case $\sigma_{ij}(L_0,L)>0$ only if $L$ meets both $Q_i$ and $Q_j$. 

However $L$ can meet at most two of $Q_1,\ldots,Q_s$. Indeed, if $L\cap Q_i$ is nonempty, then either $L$ is tangent to $Q_i$ (and hence meets $Q_i$ with multiplicity at least 2) or $L$ passes through the region bounded by $Q_i$ (and hence intersects $Q_i$ at least twice). Thus $L$ is either disjoint from $Q_i$ or meets $Q_i$ with multiplicity at least 2, so B\'ezout's theorem implies that $L$ meets at most two connected components of $Q(\mb{R})$. Therefore, at most one summand on the right-hand side of Equation~\ref{eq:sumoverpairs} is positive. Furthermore, by Lemma~\ref{lem:four grates} we have $\sigma_{ij}(L_0,L,-1)\leq2$ and thus $\sigma_{ij}(L_0,L)\leq2$ for all $1\leq i<j\leq s$.  This gives the desired upper bound of 2.
\end{proof}

\bibliography{main}{}

\begin{thebibliography}{Kum23}

\bibitem[GH81]{grossharris}
Benedict~H. Gross and Joe Harris.
\newblock Real algebraic curves.
\newblock {\em Ann. Sci. \'Ecole Norm. Sup. (4)}, 14(2):157--182, 1981.

\bibitem[Kum19]{tottheta}
Mario Kummer.
\newblock Totally real theta characteristics.
\newblock {\em Ann. Mat. Pura Appl. (4)}, 198(6):2141--2150, 2019.

\bibitem[Kum23]{kummer2023signed}
Mario Kummer.
\newblock A signed count of 2-torsion points on real abelian varieties.
\newblock {\em arXiv preprint arXiv:2301.10621}, 2023.

\bibitem[LV21]{LV21}
Hannah Larson and Isabel Vogt.
\newblock An enriched count of the bitangents to a smooth plane quartic curve.
\newblock {\em Res. Math. Sci.}, 8(2):Paper No. 26, 21, 2021.

\bibitem[MPS22]{MPS22}
Hannah Markwig, Sam Payne, and Kris Shaw.
\newblock Bitangents to plane quartics via tropical geometry: rationality,
  $\mathbb{A}^1$-enumeration, and real signed count.
\newblock {\em arXiv preprint arXiv:2207.01305}, 2022.

\end{thebibliography}
\bibliographystyle{alpha}
\end{document}